\documentclass[letterpaper, 11pt]{amsart}
\usepackage{mathtools}
\usepackage{amsmath}
\usepackage{amssymb}
\usepackage{yhmath}
\usepackage{graphicx}
\usepackage{mathrsfs}
\usepackage{bbm}
\usepackage{xcolor}
\usepackage{tikz-cd}
\usepackage{tikz}
\usetikzlibrary{patterns}
\usepackage{hyperref}

\usepackage{verbatim}
\usepackage{float}

\usepackage[normalem]{ulem} 

\DeclareMathAlphabet{\mathpzc}{OT1}{pzc}{m}{it}

\usepackage{thmtools}
\usepackage{thm-restate}

\usepackage{caption}

\newtheorem{theorem}{Theorem}[section]

\newtheorem*{claim*}{Claim}

\newtheorem{lemma}[theorem]{Lemma}
\newtheorem{lem}[theorem]{Lemma}

\newtheorem{cor}[theorem]{Corollary}

\newtheorem{prop}[theorem]{Proposition}

\theoremstyle{definition}
\newtheorem{definition}[theorem]{Definition}
\newtheorem{Def}[theorem]{Definition}

\theoremstyle{remark}

\newtheorem{Rmk}[theorem]{Remark}

\numberwithin{equation}{section}

\newcommand{\op}{\operatorname}

\newcommand{\be}{\begin{equation}}
\newcommand{\ee}{\end{equation}}
\newcommand{\Ga}{\Gamma}

\newcommand{\Z}{\mathbb Z}

\newcommand{\ga}{\gamma}

\newcommand{\la}{\lambda}
\newcommand{\La}{\Lambda}
\newcommand{\inte}{\op{int}}

\newcommand{\cal}{\mathcal}
\newcommand{\br}{\mathbb R}
\newcommand{\SO}{\op{SO}}

\newcommand{\PSL}{\op{PSL}}
\newcommand{\F}{\cal F}

\newcommand{\bH}{\mathbb H}

\newcommand{\G}{\Gamma}

\renewcommand{\frak}{\mathfrak}

\newcommand{\e}{\varepsilon}

\newcommand{\fa}{\mathfrak a}

\renewcommand{\i}{\op{i}}

\renewcommand{\S}{\mathbb S}

\newcommand{\id}{\op{id}}

\renewcommand{\P}{\mathbb{P}}

\newcommand{\fg}{\frak g}

\newcommand{\ft}{\F_\theta}

\renewcommand{\epsilon}{\e}
\renewcommand{\t}{\theta}

\newcommand{\SL}{\op{SL}}

\renewcommand{\t}{\theta}
\title[Bi-Lipschitz rigidity]{
Bi-Lipschitz rigidity  of discrete subgroups }

\author{Richard Canary}
\address{Department of Mathematics,  University of Michigan, Ann Arbor, MI}
\email{canary@umich.edu}

\author{Hee Oh}
\address{Department of Mathematics, Yale University, New Haven, CT 06511 and Korea Institute for Advanced Study, Seoul}
\email{hee.oh@yale.edu}

\author{Andrew Zimmer}
\address{Department of Mathematics, University of Wisconsin-Madison, Madison, WI}
\email{amzimmer2@wisc.edu}

\thanks{Canary is partially supported by the NSF grant No. DMS-2304636.
Oh is partially supported by the NSF grant No. DMS-1900101. Zimmer is partially supported by a Sloan research fellowship and NSF grant No. DMS-2105580.}

\begin{document}

\maketitle
\begin{abstract} We obtain a bi-Lipschitz rigidity theorem for a Zariski dense discrete subgroup of a connected simple real algebraic group. As an application, we show that any Zariski dense
discrete subgroup of a higher rank semisimple algebraic group $G$ cannot have a $C^1$-smooth slim limit set in $G/P$ for any non-maximal parabolic subgroup $P$.
\end{abstract}

\section{Introduction}

For $i=1,2$, let $G_i$ be a connected simple real algebraic group and $\Ga_i$ a Zariski dense discrete subgroup of $G_i$. Let $$\rho:\Ga_1\to \Ga_2$$ be an isomorphism. The classical rigidity problem  searches for a condition on $\rho$ which guarantees that  $\rho$ is algebraic, that is, it extends to a Lie group isomorphism $G_1\to G_2$. 

If $\Ga_1$ is a lattice in $G_1$ and either 
\begin{itemize}
\item $G_1=G_2$ has rank one and is not locally isomorphic to $\PSL_2(\br) $, or
    \item $G_1$ has higher rank,   
\end{itemize} 
 then
 {\it any} isomorphism $\rho:\Ga_1\to \Ga_2$ is algebraic by celebrated theorems of Mostow, Prasad, and Margulis (\cite{Mo}, \cite{Pr}, \cite{Ma}). On the other hand, there are very few rigidity theorems for non-lattice discrete subgroups, especially in higher rank.
In this article, 
we provide a rigidity criterion $\rho:\Ga_1\to \Ga_2$ in terms of a $\rho$-boundary map  between the limit sets of $\Ga_1$ and $\Ga_2$.

Since $\Ga_i$ is Zariski dense, there exists a unique $\Ga_i$-minimal subset $\La_i$
in $\F_i=G_i/P_i$ for a  parabolic subgroup $P_i$ of $G_i$, called the limit set. When both parabolic subgroups are maximal, our result takes the following simple form:
\begin{theorem}[Bi-Lipschitz rigidity theorem I] \label{m00} 
Assume that  $P_1$ and $P_2$ are maximal parabolic subgroups.
Let  $\rho:\Ga_1\to \Ga_2 $ be an isomorphism. If there exists a bi-Lipschitz $\rho$-equivariant map $f:\La_1\to \La_2$, then $\rho$ extends to a Lie group isomorphism $$\bar \rho: G_1\to G_2$$ which induces a diffeomorphism $\bar f: \F_1\to \F_2$ such that $\bar f|_{\La_1}=f$.\end{theorem}

Recall that  $f:\La_1\to \La_2$  is bi-Lipschitz if there exists 
$C\ge 1$ such that  for all $\xi, \eta\in \La_1$,
 \be\label{biL} C^{-1} d_{\F_1}(\xi, \eta) \le d_{\F_2} (f(\xi), f( \eta)) \le  C d_{\F_1}(\xi, \eta)\ee 
 where $d_{\F_i}$ is a Riemannian metric on $\F_i$ for $i=1,2$. Since any two Riemannian metrics on $\F_i$ are bi-Lipschitz equivalent to each other, this notion is well-defined. We note that there can be at most one $\rho$-equivariant map $f:\La_1\to \La_2$ \cite[Lemma 4.5]{KO3}.   We emphasize that we do not require $f$ to be defined on all of  $\F_1$, but only on $\La_1$.
 For $G_1=G_2=\SO(n,1)^\circ$, $n\ge 2$, Theorem \ref{m00} was proved by Tukia \cite[Theorem D]{Tu2}.
 
\begin{Rmk}
 \begin{enumerate}  \item 
The hypothesis that $G_1$ and $G_2$ are simple is necessary; see Remark \ref{semi}.
\item The global bi-Lipschitz hypothesis on $f$ can be replaced by the condition that
$f$ is bi-Lipschitz on some non-empty open subset of $\La_1$; see Lemma \ref{global}.
 \end{enumerate} 
\end{Rmk}

We now state a general version of Theorem \ref{m00}
where  $P_1$ and $P_2$ are arbitrary parabolic subgroups.
\begin{theorem}[Bi-Lipschitz rigidity theorem II] \label{m0} 
 Let
  $\rho:\Ga_1\to \Ga_2 $ be an isomorphism. 
If there exists a bi-Lipschitz $\rho$-equivariant map $f:\La_1\to \La_2$, 
then $\rho$ extends to a Lie group isomorphism 
$$\bar \rho: G_1\to G_2.$$

Moreover, there exists a parabolic subgroup $P_2'$ of $G_2$ containing $P_2$
  such that $\bar\rho(P_1)\subset P_2'$ up to a conjugation and
 the smooth submersion $G_1/P_{1}\to G_2/P_{2}'$ induced by $\bar\rho$ coincides
 with the composition $\pi\circ f$ on $\La_{1}$ where $\pi:G_2/P_{2}\to G_2/P_{2}'$ is the canonical factor map.
\end{theorem}

$$\begin{tikzcd}
    \La_1 \arrow[r, "f"] \arrow[d, hook] \arrow[rd, phantom, "\circlearrowleft"] & \La_2 \arrow[d, "\pi"] \\
    G_1/P_1 \arrow[r, " \bar \rho"'] & G_2/P_2'
\end{tikzcd}$$

See Theorem \ref{rigid2} for a stronger version which relaxes the bi-Lipschitz condition to a $\kappa$-bi-H\"older condition for $\kappa>0$.
\begin{Rmk}
 In general, $P_2'$ is not the same as $P_2$.   We use the theory of hyperconvex subgroups to construct a Zariski dense discrete subgroup of $\SL_8(\br)$
 which demonstrates this point in Proposition \ref{exam}.\end{Rmk}

Theorem \ref{m0} also has consequences for the regularity of the limit set 
of $\Ga$ in $G/P$ when $G$ is a higher rank semisimple real algebraic group and
$P$ is a non-maximal parabolic subgroup. 

\begin{theorem}[Regularity of slim limit sets] \label{nm}
    Let $G$ be a connected semisimple real algebraic group of rank at least $2$ and $P$  a non-maximal parabolic subgroup of $G$.  Any Zariski dense discrete subgroup of $G$ cannot have a slim limit set in $G/P$ which is a $C^1$-submanifold.
\end{theorem}
 Note that any non-maximal parabolic subgroup  $P$ is contained in at least two non-conjugate maximal parabolic subgroups of $G$.
 We call a subset $S\subset G/P$
{\it slim} if there exists a pair of non-conjugate maximal parabolic subgroups $P_1, P_2$ containing $P$ such that the canonical factor map $\pi_i:G/P\to G/P_i$ is injective on $S$ for $i=1,2$.

$$\begin{tikzcd}[column sep=tiny]
    & G/P \arrow[ld, "\pi_1"'] \arrow[rd, "\pi_2"] & \\
    G/P_1 & & G/P_2
\end{tikzcd}$$
In particular, the limit set of any subgroup of a $P$-Anosov or relatively $P$-Anosov subgroup is always slim. More
generally, if any two points in the limit set are in general position, then the limit set is slim.

The non-maximal hypothesis on $P$ in Theorem \ref{nm} is necessary, as there are many Zariski dense discrete subgroups of $\PSL_n(\br)$, $n\ge 3$, whose limit sets are $C^1$-submanifolds of $\P(\br^n)$, e.g., images of Hitchin \cite{La} and Benoist representations \cite{Ben2}. 
We remark that the limit sets of these examples are not $C^2$ as shown by Zimmer \cite{Zi}.  

\begin{Rmk}\begin{enumerate}
 \item When $G$ is of rank one, the limit set $\La$ of a Zariski dense subgroup of $G$ is not a proper $C^r$-submanifold of $G/P$
where $r=1$ for $G=\SO(n,1)^\circ$ and $r=2$
for other rank one groups (\cite[Proposition 3.12 and Corollary 3.13]{Wi}). 
In higher rank,
there exists $0<r<\infty$, depending on $G$, such that  $\La$ is not a proper  $C^r$-submanifold of $G/P$ for any parabolic subgroup $P$ \cite[Lemma 2.11]{ELO}. 

\item Theorem \ref{nm} was previously established for images of Hitchin representations
\cite[Corollary 6.1]{Tsouvalas} and  for images of
$(1,1,2)$-hyperconvex representation of a 
surface group \cite[Corollary 7.7]{PS}.
We also mention \cite{GM}, \cite{ZZ}, and \cite{PSW2} for related work on the regularity of the limit set for certain classes of subgroups of $G=\SO(d,2)$, $\PSL_d(\br)$ and $\SO(p,q)$ respectively.  
\end{enumerate}
\end{Rmk}

\subsection*{On the proofs}
We deduce Theorem \ref{m0} and Theorem \ref{nm} from the following property of limit sets of a Zariski dense subgroup in higher rank:
 \begin{prop}\label{m1}
     Let $G$ be a connected semisimple real algebraic group of rank at least $2$. 
      Let  $Q_1$ and $Q_2$  be a pair of parabolic subgroups of $G$ such that
      there is no parabolic subgroup of $G$ containing $Q_1$ and a conjugate of $Q_2$ (e.g., a pair of non-conjugate maximal parabolic subgroups).
     
      If $\Ga<G$ is a Zariski dense
     discrete subgroup,
     then there is no $\Ga$-equivariant bi-Lipchitz map between the limit sets of $\Ga$ on $G/Q_1$ and $G/Q_2$.
      \end{prop}

Indeed, if $\rho$ in Theorem \ref{m0} does not extend to a Lie group isomorphism $G_1\to G_2$, then the following self-joining subgroup
\be\label{self} \Ga=(\id\times \rho)(\Ga_1)=\{(g, \rho(g)):g\in \Ga_1\}\ee  is a Zariski dense subgroup
of the product $G=G_1\times G_2$. On the other hand,  a bi-Lipschitz  map $f$ as in Theorem \ref{m0} yields a bi-Lipschitz homeomorphism between  the limit sets
of the self-joining group $\Ga$ in $G/(P_1\times G_2)$ and $G/(G_1\times P_2)$, which then gives a desired contradiction by Proposition \ref{m1}. 
We mention the recent work \cite{KO} and \cite{KO2} on related rigidity theorems
which use the idea of self-joinings.

If $\Ga$ has a $C^1$-slim limit set in $G/P$ as in Theorem \ref{nm} and $P_1$ and $P_2$ are non-conjugate maximal parabolic subgroups containing $P$, we get a bi-Lipschitz map between the limit sets of
$\Ga$ in $G/P_1$ and $G/P_2$ from the slimneess hypothesis. Therefore Proposition \ref{m1} implies Theorem \ref{nm}.

For the proof of Proposition \ref{m1}, we relate the exponential contraction rates of loxodromic elements $\gamma\in \Ga$ on $G/Q_i$ with the Jordan projections of the image of $\ga$ under Tits representations of $G$. This part of the argument is motivated by earlier work of Zimmer \cite[Section 8]{Zi}.
We then show that the bi-Lipschitz equivalence of the limit sets gives an obstruction to Benoist's theorem \cite{Ben} on
the non-empty interior property of
the limit cone  of a Zariski dense subgroup
(see the proof of Proposition \ref{no}).

\subsection*{Acknowledgement} We would like to thank Dongryul Kim for helpful comments on a preliminary version of this paper.

\section{Preliminaries}
Unless mentioned otherwise, let $G$ be a connected semisimple {\it real} algebraic group throughout the paper. This means that
$G$ is the identity component $\bf G(\br)^\circ$ for a 
 semisimple algebraic group $\bf G$  defined over $\br$. A parabolic $\br$-subgroup
$\mathbf P$ of $\bf G$ is a proper algebraic subgroup defined over $\br$ such that the quotient $\bf G /\bf P$ is a projective  algebraic variety.
A parabolic subgroup
$P$ of $G$ is of the form $\mathbf P(\br)$ for a parabolic $\br$-subgroup $\mathbf P$ of $\bf G$; in this case, the quotient $G/P$ is equal to $(\bf G/\bf P)(\br)$ and is a real projective variety, called a $G$-boundary \cite{Bo}. 
Any parabolic subgroup $P$ is conjugate to a unique standard parabolic subgroup of $G$, once we fix a root system associated to $G$.

To be precise,  let $A$ be a maximal real split torus of $G$. The rank of $G$
is defined as the dimension of $A$.
Let $\fg$ and $\fa$ respectively denote the Lie algebras of $G$
and $A$. Fix a positive Weyl chamber $\fa^+ \subset \fa$ and set $A^+=\exp \fa^+$, and a maximal compact subgroup $K< G$ such that the Cartan decomposition $G=K A^+ K$ holds. We denote by $M$ the centralizer of $A$ in $K$.
For $g\in G$, we denote by $\mu(g)$ the Cartan projection of $g$, which is the unique element of $\fa^+$ such that $g\in K \exp \mu(g) K$.

Any $g\in G$ can be written as the commuting product $g=g_hg_e g_u$ where $g_h$ is hyperbolic, $g_e$ is elliptic and $g_u$ is unipotent. 
	The hyperbolic component $g_h$ is conjugate to a unique element $\exp \lambda(g) \in A^+$ and \be\label{Jordan} \lambda(g)\in \fa^+\ee  is called 
	the Jordan projection of $g$.
	When $\lambda(g)\in \inte \fa^+$, $g\in G$ is called {\it{loxodromic}} in which case $g_u$ is necessarily trivial and $g_e$ is conjugate to
	an element $m\in M$.
 
Let $\Phi=\Phi(\fg, \fa)$ denote the set of all roots and $\Pi$  the set of all  simple roots given by the choice of $\fa^+$. 
 The Weyl group $\cal W$ is given by $N_K(A)/M$ where  $N_K(A)$ is  the normalizer of $A$ in $K$.
 
Consider the real vector space $\mathsf E^*=\mathsf X(A)\otimes_\Z \br$
where $\mathsf X(A)$ is the group of all real characters of $A$ and let $\mathsf E$ be its dual.
Denote by $( \cdot,\cdot )$ a $\cal W$-invariant inner product on $\mathsf E$.
 We
denote by $\{\omega_{\alpha}:\alpha\in \Pi\}$ the (restricted) fundamental weights of $\Phi$ defined by $$ 2\frac{(\omega_\alpha, \beta)}{(\beta,\beta)} = c_\alpha \delta_{\alpha, \beta}$$ where $c_\alpha=1$ if $2\alpha\notin \Phi$ and $c_\alpha=2$
otherwise.

 Fix an element $w_0\in N_K(A) $ of order $2$  representing the longest Weyl element so that $\op{Ad}_{w_0}\mathfrak a^+= -\mathfrak a^+$. 
The map $$\i= -\op{Ad}_{w_0}:\fa\to \fa $$ is called the opposition involution.
It induces an involution of $ \Phi$ preserving $\Pi$, for which we use the same notation $\i$, so that $\i (\alpha )  = \alpha \circ \i$ for all $\alpha\in \Phi$.

For a  non-empty subset $\theta$ of $ \Pi$, 
let $\fa_\theta =\cap_{\alpha\in \Pi-\theta} \ker \alpha$, and
let $ P_\theta$ denote a standard parabolic subgroup of $G$ corresponding to $\theta$; that is,
$P_{\theta}=L_\theta N_\theta$ where $L_\theta$ is the centralizer of $\exp \fa_\theta$  and $N_\theta$
is the unipotent radical of $P_\theta$ which is generated by
root subgroups associated to all positive roots which are not $\mathbb Z$-linear combinations of elements of $\Pi-\theta$. If $\theta=\Pi$, then $P=P_\Pi$ is a minimal parabolic subgroup.  For a singleton $\theta=\{\alpha\}$,
$P_\alpha$ is a maximal parabolic subgroup of $G$. Any parabolic subgroup $P$ is conjugate to
a unique standard parabolic subgroup $P_\theta$ for some non-empty subset $\theta\subset \Pi$.

We consider the $\theta$-boundary: $$\F_\theta=G/P_\theta.$$
We denote by $d_{\F_\theta}$ a Riemannian metric on $\F_\theta$.
 Let  $P_\theta^+ =w_0 P_{\i(\theta)}w_0^{-1}$, which is the standard parabolic subgroup opposite to $P_\theta$ such that $P_\theta\cap P_\theta^+=L_\theta$.
 Hence $\F_{\i(\theta)}= G/P_{\i(\theta)}=G/P_\theta^+.$
The $G$-orbit $\F_\theta^{(2)}=\{ (g P_\theta, gw_0P_{\i(\theta)}):g\in  G\}$ is the unique open $G$-orbit in $G/P_\theta\times G/P_\theta^+$ under the diagonal $G$-action.
 Two elements
$\xi\in \F_\theta$ and $\eta\in \F_{\i(\theta)}$ are said to be in general position if $(\xi, \eta)\in \F_\theta^{(2)}$.

\section{Contraction rates of loxodromic elements and Tits representations}

The first part of the following theorem immediately follows as a special case of a theorem of Tits \cite{Ti}, and the second part is remarked in \cite{Ben} and proved in \cite{Sm}.

\begin{theorem}[{\cite[Theorem 7.2]{Ti}, \cite[Lemma 2.13]{Sm}}] \label{tits} 
  For each $\alpha\in \Pi$, there exists an irreducible representation ${\rho}_{\alpha}: {G}\to \op{GL} (V_{\alpha}) $  whose  highest (restricted) weight 
    $\chi_{\alpha}$ is equal to $k_\alpha \omega_\alpha$
    for some positive integer $k_\alpha$ and whose highest weight space is one-dimensional.
   
    Moreover, all weights of $\rho_{\alpha}$ are $\chi_\alpha$, $\chi_\alpha-\alpha$ and
  weights of the form $\chi_{\alpha} - \alpha - \sum_{\beta \in \Pi} n_{\beta} \beta$ with $n_\beta$ non-negative integers. 
\end{theorem}
These representations are called Tits representations of $G$. Fix $\alpha \in \Pi$ and, as before, set $\F_\alpha=G/P_\alpha$. We denote by $V_1$ and $V_2$ the weight spaces of $\rho_\alpha$ for the highest weight $\chi_\alpha$ and  the second highest weight $\chi_\alpha -\alpha$ respectively. We have $\dim V_1=1$
and $\dim V_2\ge 1$. 
 If we set $\xi_\alpha=[P_\alpha]\in \F_\alpha$,
  the map $g\xi_\alpha  \mapsto  gV_1$ gives an embedding \be\label{pi} \F_\alpha \to  \P(V_{\alpha})\ee 
whose image is a closed subvariety. We may hence identify $\F_\alpha$ as a closed subvariety of $\P(V_\alpha)$.
Let $\langle \cdot, \cdot \rangle_{\alpha}$ be a $K$-invariant inner product on $V_{\alpha}$ with respect to which $A$ is symmetric and we have the orthogonal weight space decomposition of
$V_\alpha$. Using the norms on $V_\alpha$ and $\wedge^2 V_\alpha$ induced by this inner product, we get a $K$-invariant Riemannian metric $d_\alpha$ on $\P(V_\alpha)$:
$$d_\alpha ([v], [w])=\frac{\|v\wedge w\|}{\|v\| \|w\|}\quad \text{ for $[v], [w]\in \P(V_\alpha)$}.$$

Recall that an element $g\in G$ is {\it loxodromic} if there exist $a\in \inte A^+$  and $m\in M$ such that
$g= h_g am h_g^{-1}$ for some $h_g\in G$. The element $h_g$ is then uniquely determined modulo $AM$ and $\lambda(g)=\log a\in \inte \fa^+$. 

Let $\pi_i=\pi_{\alpha, i}: V_\alpha\to V_i$ be the orthogonal projection for $i=1,2$.
Recall the following standard lemma:
\begin{lem}\label{att} Let $g $
be a loxodromic element of $G$. For $\xi\in \F_\alpha$, we have $\pi_{1}(h_g^{-1} \xi)\ne 0$ if and only if
$g^n \xi$ converges to $ h_g\xi_\alpha $ as $n\to \infty$. 
\end{lem} The point $y_\alpha^g: =h_g \xi_\alpha \in \F_\alpha$ is called the attracting fixed point of $g$.

\begin{lem} \label{lox} Let $g\in G$ be a loxodromic element and $\alpha\in \Pi$.
\begin{enumerate}
    \item For all $\xi\in \F_\alpha$ with $\pi_{ 1}(h_g^{-1} \xi)\ne 0$, we have
$$ -\alpha (\lambda (g))\ge \limsup_{n\to \infty}\frac{1}{n}\log d_\alpha (g^n \xi, y_\alpha^g). $$

\item  For all $\xi\in \F_\alpha$ with $\pi_{ 1}(h_g^{-1}\xi)\ne 0$ and $\pi_{ 2}(h_g^{-1}\xi)\ne 0$,  we have
$$ -\alpha (\lambda (g))= \lim_{n\to \infty}\frac{1}{n}\log d_\alpha (g^n \xi, y_\alpha^g). $$

\end{enumerate}
\end{lem}

\begin{proof}
It suffices to prove the claim when $h_g=e$, i.e., $g=am\in AM$ with $\log a\in \inte \fa^+$. 
  Considering $\xi\in \F_\alpha\subset \P(V_\alpha)$, choose a vector $v\in V_\alpha$ representing $\xi$.  List  all distinct weights
  of $\rho_\alpha$ given by Theorem \ref{tits} as follows:
  $\chi_1=\chi_\alpha$, $\chi_2=\chi_\alpha-\alpha$, and
  $\chi_i=\chi_\alpha -\alpha -\beta_i$, $3\le i \le \ell $; in particular,  $\beta_i\ne 0$ is a non-negative integral linear combinations of simple roots. Let $V_i$ denote the weight space corresponding to $\chi_i$ and
 write $v=v_1+v_2+ \cdots + v_\ell$ so that $v_i\in V_i$ for each $1\le i\le \ell$.
  Suppose that $\pi_1(\xi)\ne 0$, that is $v_1 \neq 0$. 
  We may then assume that $v_1$ is a unit vector relative to $\langle \cdot, \cdot \rangle_{\alpha}$.
  Since $M$ commutes with $A$, $M$ stabilizes each weight subspace, and in particular, $M v_1=\pm v_1$.
  Now $$g^n v=  e^{n \chi_{\alpha}(\log a)} m^n v_1 +  e^{n (\chi_{\alpha}-\alpha) (\log a)} m^n v_2
  + \sum_{i=3}^\ell e^{n (\chi_{\alpha}-\alpha-\beta_i) (\log a)} m^n v_i .$$
 
  Hence the projection $p(g^nv)$ of $g^n v$ to the affine chart $\mathbb A=\{w\in V_\alpha: \pi_1(w)=v_1\}$ is
  $$p(g^nv)= v_1 +  e^{-n \alpha (\log a)} m^n v_2'
  + \sum_{i=3}^\ell  e^{- n( \alpha+\beta_i) (\log a)} m^n v_i'$$ where $v_i'=\pm v_i$, depending on the sign of $m^n v_1$.
Note that $\lim g^n \xi = V_1$, and that the metric $d_\alpha$ on a neighborhood on $V_1$ in $\P(V_\alpha)$
  is bi-Lipschitz equivalent to the metric $d$ on the affine chart $\mathbb A$, obtained by restricting the distance on $V_\alpha$ induced by $\langle \cdot, \cdot \rangle_{\alpha}$.
  
Since the weight spaces are orthogonal, we have
  $$d ( p(g^n v), v_1)=  e^{-  n \alpha (\log a)} \left( \| v_2\|^2
  + \| w_n\|^2 \right)^{1/2}  $$
  where $w_n= \sum_{i=3}^\ell e^{- n\beta_i (\log a)} m^n v_i'$ and $\| \cdot \|$ is the norm induced by $\langle \cdot, \cdot \rangle_{\alpha}$.
    Since $\log a\in \inte \fa^+$ and hence $\beta_i(\log a)>0$ for all $3\le i\le \ell$, we have
    $$\lim_{n\to \infty} w_n = 0 .$$
    
     First consider the case when
     $\pi_2(\xi)= 0$, that is $v_2= 0$. Since $\log \|w_n\|<0$ for all large $n$, we have
   \begin{multline*}
        \limsup_{n\to \infty}  \frac{1}{n} \log d_\alpha (g^n\xi , y_\alpha^g) = \limsup_{n\to \infty}  \frac{1}{n} \log d(p(g^nv), v_1)\\
        =  \limsup_{n\to \infty} \frac{1}{n} ( -n\alpha (\log a) +\log \|w_n\|) \le -\alpha (\log a) .   \end{multline*} 

Now suppose that $\pi_2(\xi)\ne 0$, that is $v_2\ne 0$. Again since $w_n \to 0$, 
we have
$$  \lim_{n\to \infty}  \frac{1}{n} \log d_\alpha (g^n\xi , y_\alpha^g)  =\lim_{n\to \infty} \frac{1}{n} \log d(p(g^nv), v_1) = -\alpha(\log a) .$$
This finishes the proof.
\end{proof}

\section{Bi-Lipschitz rigidity of discrete subgroups}\label{sec:biLip rigidity...}

Let $G$ be a connected semisimple real algebraic group and $X=G/K$ be the associated Riemannian symmetric space and fix $o=[K]\in X$.

We consider the following notion of convergence of a sequence in $G$ to an element of $\F_\theta=G/P_\theta$ for a non-empty subset $\theta\subset \Pi$.

For a sequence $g_io \in X$  and $\xi\in \ft$, we write $\lim g_i o =\xi$ and
 say  $g_io \in X$ converges to $\xi$ if \begin{enumerate}
     \item $\min_{\alpha\in \theta} \alpha(\mu(g_i)) \to \infty$ as $i\to \infty$; and
\item $\lim_{i\to\infty} \kappa_{g_i}P_\theta= \xi$ in $\F_\theta$ for some $\kappa_{g_i}\in K$ such that $g_i\in \kappa_{g_i} A^+ K$.
 \end{enumerate}

\begin{definition}\label{l2} Let $\Gamma<G$ be a  discrete subgroup
and let $\F=G/P$ for a parabolic subgroup $P$. Let $\theta\subset \Pi$ be a unique subset such that $P$ is conjugate to $P_\theta$ and hence $\F=\F_\theta$.
The limit set of $\Ga$ in $\F_\theta$ is then defined as the set of all accumulation points of $\Ga(o)$ in $\F_\theta$:
$$ \La_\theta=\Lambda_\theta(\Ga)=\{\lim \ga_i(o)\in \F_\theta: \ga_i \in \Ga\}.$$
\end{definition}
It is a $\Ga$-invariant closed subset of $\F_\theta$, which is non-empty provided
$\Ga$ contains a sequence $\ga_i$ satisfying $\lim_{i\to \infty} \min_{\alpha\in \theta} \alpha(\mu(\ga_i)) = \infty$. 
If $\Ga$ is Zariski dense, $\La_\theta$ is the unique $\Ga$-minimal subset of $\F_\theta$ and can also be described as  the set of all $\xi\in \F_\theta$ such that
the Dirac measure $\delta_\xi$ is the weak limit of $(\ga_i)_* \op{Leb}_\theta$ for some sequence
$\ga_i\in \Ga$ where $\op{Leb}_\theta$ denotes the unique $K$-invariant probability measure on $\F_\theta$ (\cite{Ben}, \cite{Qu}). Moreover, if $\Theta\subset \theta$, then $\La_\Theta$ is equal to the image of $\La_\theta$ under the canonical projection $\F_\theta\to \F_\Theta$, by  minimality.

The limit cone of $\Ga$ is defined as the smallest closed cone of $\fa^+$ containing all Jordan projections of loxodromic elements of $\Ga$. 
\begin{theorem}[Benoist \cite{Ben}]\label{ben}
    If $\Ga<G$ is Zariski dense, its limit cone has non-empty interior in $\fa$.
\end{theorem}

For $\kappa>0$ and $\theta_1, \theta_2\subset \Pi$,
a map $F:\La_{\theta_1}\to \La_{\theta_2}$ is called $\kappa$-bi-H\"older if there exists $C>0$ 
such that for all $x,y\in \La_{\theta_1}$  \be\label{ck} C^{-1} d_{\F_{\theta_1}} (x, y)^{\kappa} \le  d_{\F_{\theta_2}} (F(x),F(y))\le C d_{\F_{\theta_1}} (x, y)^{\kappa} \ee
where $d_{\F_{\theta_i}}$ is a Riemannian metric on $\F_{\theta_i}$ for $i=1,2$.
Observe that if $\Ga$ is Zariski dense, any $\Ga$-equivariant $\kappa$-bi-H\"older map \hbox{$\La_{\theta_1}\to \La_{\theta_2}$} is a homeomorphism; the minimality of $\La_{\theta_2}$ implies the surjectivity and the bi-H\"older property implies the injectivity. Therefore  $F$ is
$\kappa$-bi-H\"older if and only if  $F$ is $\kappa$-H\"older
and $F^{-1}$ is $\kappa^{-1}$-H\"older.

 Proposition \ref{m1} follows from the following for $\kappa=1$:
\begin{prop}\label{no} Let $\Ga<G$ be Zariski dense. Let $\theta_1$ and $\theta_2$ be disjoint non-empty subsets of $\Pi$.
Then for any $\kappa>0$, there exists no $\Ga$-equivariant  $\kappa$-bi-H\"older map $F:\La_{\theta_1}\to \La_{\theta_2}$.   

 
\end{prop}

\begin{proof} For simplicity, we write $\La_i=\La_{\t_i}$ and
$d_{\theta_i}=d_{\F_{\theta_i}}$.
   Let $F:\La_1\to \La_2$ be a $\Ga$-equivariant  homeomorphsim. Fix $\kappa>0$.
   Since $\theta_1\cap \theta_2=\emptyset$, the union $\bigcup_{\alpha_1\in \theta_1, \alpha_2\in \theta_2} \op{ker}(\kappa \alpha_1-\alpha_2)$ is a finite union of hyperplanes of $\fa$. Therefore by Theorem \ref{ben}, $\Ga$ contains a loxodromic element
    $\ga$ such that
    $$ \{\kappa \cdot \alpha(\la(\ga)):\alpha\in \t_1\}\cap \{ \alpha(\la(\ga)):\alpha\in \t_2\}=\emptyset.$$

For each $i=1,2$, let $\alpha_i\in \theta_i$ be such that 
\be\label{aii} \alpha_i(\lambda(\ga))=\min\{\alpha(\lambda(\ga)):\alpha\in \t_i\}.\ee 
Note that \be\label{choice} \kappa \cdot \alpha_1(\lambda(\ga))\ne \alpha_2(\lambda(\ga)) .\ee

\noindent{\bf Claim:} If $F^{-1}$ is $\kappa^{-1}$-H\"older,
  then 
    \be\label{al1}  \alpha_2(\la(\ga))\le  \kappa\cdot  \alpha_1(\la(\ga)).\ee

    By replacing $\Ga$ by a suitable conjugate,
    we may also assume that $\ga =am\in \Ga$ with $a\in \inte A^+$ and $m\in M$. For each $i=1,2$, let $y_i=y_{\alpha_i}^\ga$
    denote the attracting fixed point of $\ga$ in $\cal F_i$; we have $y_i\in \La_i$.
  As $\Ga$ is Zariski dense, $\La_i$ is Zariski dense in $\cal F_i$ for each $i=1,2$.  
   Let $\pi_{\alpha,1}$ and $ \pi_{\alpha,2}$ be as in Lemmas~\ref{att} and~\ref{lox} for each $\alpha\in \Pi$. Since the set $$\cal O=\{\xi\in \cal F_1: \pi_{\alpha, 1}(\xi)\ne 0, \pi_{\alpha, 2}(\xi)\ne 0 \text{ for all $\alpha\in \theta_1$}\}$$ is a Zariski open subset of $\cal F_1$, 
 the intersection $\cal O\cap \La_1$ is a non-empty open subset of $\La_1$.
As $F$ is a homeomorphism,  the image $F (\cal O\cap \La_1)$ is  a non-empty open subset of $\La_2$. Since $Z=\{\xi\in \F_2: \ga^n \xi\not \to y_2\text { as $n\to \infty$}\}$ is a proper Zariski closed subset of $\F_2$ by Lemma \ref{att}, 
$F(\cal O\cap \La_1) $ cannot be contained in $Z$; otherwise it would imply that $\La_2$ is contained in a proper Zariski closed subset by the $\Ga_2$-minimality of $\La_2$, which contradicts the Zariski density of $\Ga_2$.
Therefore there exists an element $\xi\in \cal O\cap \La_1$ such that  $\lim_{n\to \infty} \ga^n F(\xi)= y_2$.
 By the equivariance and continuity of $F$, we have
   \be\label{yy} F(y_1)=\lim F(\ga^n \xi)= \lim \ga^n F(\xi)=y_2.\ee 

Let $i=1,2$. Since $P_{\theta_i}=\bigcap_{\alpha\in \theta_i} P_{\alpha}$, 
we have a diagonal embedding 
$$\F_i=G/P_{\theta_i}\to  \prod_{\alpha\in \theta_i}\mathbb P(V_\alpha) $$
via the product of the maps in \eqref{pi}.
Consider the metric $d_i$ on $\F_i$ obtained as the restriction
of $\sum_{\alpha\in \theta_i} d_{\alpha}$ to $\F_i$: for $\eta=gP_{\theta_1}$
and $\eta'=g'P_{\t_2}$ with $g, g'\in G$,
$$d_i(\eta, \eta')=\sum_{\alpha\in \theta_i} d_{\alpha} (\eta, \eta') $$
where
$d_{\alpha}(\eta, \eta'):= d_{\alpha} (gV_{\alpha, 1}, g'V_{\alpha,1}) $
where $V_{\alpha,1}$ is the highest weight line of $\rho_\alpha$ as in \eqref{pi}.
Since $d_i$ is bi-Lipschitz equivalent to a Riemannian metric on $\F_i$, we have that
 $F^{-1}: (\La_2, d_{2})\to (\La_1, d_{1})$ is $\kappa^{-1}$-H\"older. 

Since $\xi\in \cal O$ and  $\lim \ga^n F(\xi)= y_2$,  we have  by Lemma \ref{lox} that
  $$-\alpha(\lambda(\ga))= \lim \frac{1}{n} \log d_{\alpha} (\ga^n\xi, y_1) \quad\text{for each $\alpha\in \theta_1$} $$
and $$-\alpha(\lambda(\ga))\ge \limsup\frac{1}{n} \log d_{\alpha} (\ga^n F(\xi), y_2) \text{ for each $\alpha\in \theta_2$} .$$

Since $d_{\alpha_1} (\eta, \eta')\le d_1 (\eta, \eta')$,
$d_2(\eta, \eta')\le \# \theta_2 \max_{\alpha\in \theta_2} d_\alpha (\eta, \eta')$, and $F^{-1}$ is $\kappa^{-1}$-H\"older,
we have
\begin{align}\label{las}
 -\alpha_1(\lambda(\ga)) &= \lim\frac{1}{n} \log d_{\alpha_1} (\ga^n \xi, y_1) \\  &\le
     \lim\frac{1}{n} \log d_{1} (\ga^n \xi, y_1)  
\notag \\    &\le  \kappa^{-1} \limsup\frac{1}{n} \log d_{2} ( F(\ga^n \xi), F(y_1))\notag \\ & = \kappa^{-1}\limsup\frac{1}{n} \log d_{2} (\ga^n F(\xi), y_2)  \notag \\&
= \kappa^{-1} \max_{\alpha\in \theta_2}
\limsup\frac{1}{n} \log d_{\alpha} (\ga^n F(\xi), y_2)
\notag \\ & \le -\kappa^{-1} \min_{\alpha\in \theta_2}\alpha(\lambda(\ga))=- \kappa^{-1}\alpha_2(\lambda(\ga))
 .\notag
\end{align} 

This implies that $\alpha_2(\la(\gamma))\le \kappa \alpha_1(\lambda(\ga))$, proving the claim. 

By switching the role of $\theta_1$ and $\theta_2$,
this claim then implies that 
if $F$ is $\kappa$-H\"older, then
$\alpha_1(\la(\ga))\le \kappa^{-1} \alpha_2(\la(\ga))$.
Therefore if $F$ is $\kappa$-bi-H\"older,
then $\kappa \cdot \alpha_1(\la(\ga))= \alpha_2(\la(\ga))$, contradicting \eqref{choice}.
This finishes the proof.
\end{proof}

The proof of Proposition \ref{no} shows the following as well:

\begin{prop}\label{noo}
    Let $\Ga<G$ be Zariski dense and 
    let $\theta_1, \theta_2\subset \Pi$ be non-empty
    disjoint subsets. Suppose that $\La_{\t_1}$ and $ \La_{\t_2}$ are $C^1$-submanifolds of $\F_{\theta_1}$ and $\F_{\theta_2}$ respectively.
     If $F:\La_{\t_1}\to \La_{\t_2}$ is a $\Ga$-equivariant homeomorphism, $F$ cannot be $C^1$ with non-vanishing Jacobian at any $\xi\in \cal A$, where 
      $\cal A\subset \La_{\t_1}$ is the set of all attracting fixed points
    of loxodromic elements $\ga\in \Ga$ such that $\{\alpha(\lambda(\ga)):\alpha\in \theta_1\}\cap \{\alpha(\lambda(\ga)):\alpha\in \theta_2\}=\emptyset$.
\end{prop}
\begin{proof} Let $\ga\in \Ga$ be as above.
    For each $i=1,2$, let $y_i\in \La_{\t_i}$ be the attracting fixed point of $\ga$. Then $F(y_1)=y_2$ by \eqref{yy}.
     Suppose that $F$ is $C^1$ at $y_1$, and the Jacobian of $F$ at $y_1$
    is not zero. Then $F^{-1}$ is also $C^1$ at $y_2$.
Using the exponential maps and the Taylor series expansion of $F$, we get that there exist $ c\ge 1$ and an open neighborhood $U$ of $y_1$ in $\La_{\t_1}$ such that for all $y\in U$,
    \be\label{ccc} c^{-1} d_{1}( y, y_1) \le d_{2}(F(y), F(y_1))\le c d_{1}( y, y_1) .\ee 
 Let $\alpha_i\in \t_i$ be as in \eqref{aii}. Without loss of generality, we may assume 
   $\alpha_1(\la(\ga))< \alpha_2(\la(\ga))$ by switching the indexes if necessary.  On the other hand, using \eqref{ccc}, the computation \eqref{las}
 gives $\alpha_2(\la(\gamma))\le \alpha_1(\lambda(\ga))$, which yields a contradiction.
\end{proof}

\begin{Rmk}
    It would be interesting to know whether  $\cal A$ can be replaced by the set of all {\it conical} limit points of $\Ga$ in Proposition \ref{noo}. A point $\xi=gP_{\t_1}$ is $\Ga$-conical if
    $\limsup \Gamma g (K\cap P_{\t_1})A^+\ne \emptyset$, that is,
    there exists a sequence $\ga_i\in \Ga$, $a_i\in A^+$ and $m_i\in K\cap P_{\theta_i}$ such that $\ga_i g m_i a_i$ converges (see \cite[Lemma 5.4]{KOW}
    for an equivalent definition in terms of shadows).  
    
    This question is inspired by a related result for $G=\SO(n+1,1)^\circ$. Tukia \cite{tukia}  showed that
    if $f:\S^n \to\S^n $ is a homeomorphism
    which conjugates a discrete subgroup $\Gamma_1$ of $G$
    to another discrete group $\Gamma_2$ and has a non-vanishing Jacobian at a conical limit point of $\Gamma_1$, then $\Gamma_1$ is conjugate to $\Gamma_2$ (see also \cite{ivanov} for an extension of this result to other rank one groups).
    For a related result for $(1,1,2)$-hyperconvex groups, see \cite[Corollary 7.5]{PS}.
    \end{Rmk}

In the rest of this section, let $G_i$ be a connected simple real algebraic group and
$\theta_i$ be a non-empty set of simple roots of $G_i$ for $i=1,2$.
Let $\Gamma_i<G_i$ be a Zariski dense discrete subgroup and $\La_{\t_i}$
denote the limit set of $\Ga_i$ in $\F_i=G_i/P_{\t_i}$.

\begin{lemma}  \cite[Lemma 4.5]{KO3}\label{un}
  For any isomorphism $\rho:\Ga_1\to \Ga_2$,
   there exists at most one $\rho$-equivariant continuous map $f:\La_{\t_1}\to \La_{\t_2}$.
\end{lemma}
Indeed, $f$ must send the attracting fixed point of any loxodromic element $\ga$ to that of $\rho(\ga)$ whenever $\rho(\ga)$ is loxodromic. Since the set of attracting fixed points of loxodromic elements is dense in $\La_{\t_1}$ by the Zariski density hypothesis on $\Ga_1$ \cite{Ben} and $f$ is continuous, this determines the map $f$.

Theorem \ref{m0} is a special case of the following theorem for $\kappa=1$:
\begin{theorem} \label{rigid2}
 Suppose that there exists a $\rho$-equivariant   $\kappa$-bi-H\"older map
 $f:\La_{\t_1}\to \F_2$ for some $\kappa>0$.
Then $\rho$ extends to a Lie group isomorphism  $\bar\rho :G_1\to G_2$.
 Moreover, there exists a non-empty subset 
 $\Theta_2\subset \theta_2$ such that $\bar\rho$ maps
 $P_{\theta_1}$ into a conjugate of $P_{\Theta_2}$ and
 the smooth submersion $G_1/P_{\t_1}\to G_2/P_{\Theta_2}$ induced by $\bar\rho$ coincides
 with the composition $\pi\circ f$ on $\La_{\theta_1}$ where $\pi:G_2/P_{\theta_2}\to G_2/P_{\Theta_2}$ is the canonical factor map. 
    \end{theorem}

\begin{proof}
Let $G=G_1\times G_2$.
Define the following self-joining subgroup
$$ \Ga=(\op{id}\times \rho)(\Ga_1)=\{(\ga, \rho(\ga)):\ga\in \Ga_1 \} <G.$$
Note that $P_1:=P_{\t_1}\times G_2$ and $P_2:=G_1\times P_{\t_2}$
are parabolic subgroups of $G$. The maps
$g_1P_{\t_1}\mapsto (g_1, e) P_1$ and 
$g_2P_{\t_2}\mapsto (e, g_2) P_2$ define diffeomorphisms between
$G_1/P_{\t_1}$ and $G_2/P_{\t_2}$ 
 with $G/P_1$ and $G/P_2$ respectively. Moreover,
 under this identification, the limit set
 $\La_{\t_i}$ of $\Ga_i$ in $G_i/P_{\t_i}$ corresponds to
 the limit set $\La_i$ of the self-joining $\Ga$ in $G/P_i$ for each $i=1,2$.

Since $f$ is a $\rho$-equivariant  continuous embedding  of $\La_{\t_1}$ into $ G/P_{\t_2}$,
its image is a $\Ga_2$-invariant compact subset. Since $\La_{\t_1}$ is a $\Ga_1$-minimal subset,
the image $f(\La_{\t_1})$  is also a $\Ga_2$-minimal subset. Therefore $f(\La_{\t_1})=\La_{\t_2}$ and hence
we have a $\Ga$-equivariant bijection $f:\La_1\to \La_2$ which is $\kappa$-bi-H\"older.
    
 Since $P_1$ and $P_2$ are parabolic subgroups corresponding to disjoint subsets of simple roots of $G$, Proposition \ref{no} implies that $\Ga$ cannot be Zariski dense in $G$. Since both $G_1$ and $G_2$ are simple,  the non-Zariski density of the self-joining group $\Ga$ implies that
$\rho$ extends to a Lie group isomorphism $\bar\rho: G_1\to G_2$ (cf. \cite{DK1}).

  Since $\bar\rho(P_{\t_1})$ must be a parabolic subgroup of $G_2$, there exists $g\in G_2$ such that $\bar\rho(P_{\t_1})= g P_{\t_0}g^{-1}$ where $\t_0$ is a  non-empty subset of some simple roots of $G_2$.  We claim $\t_0\cap \t_2 \ne\emptyset$. 
 By replacing $\rho$ by $\op{inn}(g)\circ \rho$
 where $\op{inn}(g):G_2\to G_2$ is the conjugation by $g$, we may assume without loss of generality that $g=e$. The isomorphism $\bar \rho$
 induces a diffeomorphism $\tilde \Phi: G_1/P_{\t_1} \to G_2/P_{\t_0}$
given by $\tilde \Phi (g_1P_{\t_1})= \bar\rho (g_1) P_{\t_0}$. Denote by $\La_{\t_0}$  the limit set of $\Ga_2$
in $G_2/P_{\t_0}$. Since $\bar\rho|_{\Ga_1}=\rho$ and hence $\tilde \Phi$ is $\rho$-equivariant, we have
$\tilde \Phi(\La_{\t_1})=\La_{\t_0}$. Then the composition $F:= f\circ \tilde \Phi^{-1}$ restricted to $\La_{\t_0}$ yields a $\kappa$-bi-H\"older map between $\La_{\t_0}$ and $\La_{\t_2}$. Since $\tilde \Phi^{-1}$ is $\rho^{-1}$-equivariant and $f$ is $\rho$-equivariant,
$F$ is $\Ga_2$-equivariant.
  So by applying Proposition~\ref{no} one more time, we obtain $\t_0\cap \t_2 \ne \emptyset$. Setting $\Theta_2=\t_0\cap \t_2$, since $P_{\theta_0}$ and $ P_{\theta_2}$
  are subgroups of $P_{\Theta_2}$, we get a map $\Phi:=G_1/P_{\t_1}\to G_2/P_{\Theta_2}$ by composing $\tilde \Phi$ with the canonical factor map $G_1/P_{\theta_0}\to G_2/P_{\Theta_2}$.
The last claim $\Phi=\pi\circ f$ on $\La_{\t_1}$ follows from Lemma \ref{un}.   This finishes the proof.
\end{proof}

\begin{Rmk}\label{semi}
      The hypothesis that $G_1$ and $G_2$ are simple is necessary in Theorem \ref{rigid2}.
    For example, consider a discrete Zariski dense subgroup $\Ga$ of a simple algebraic group $G$ with a discrete faithful representation $\rho:\Ga\to G$ which does not extend to $G$. Then $\Ga_\rho=(\text{id}\times \rho)(\Ga)$ is Zariski dense in $G$ and
    the map $\ga\to (\ga, \rho(\ga))$ gives an isomorphism
    $\Ga\to \Ga_\rho$. On the other hand,  for any parabolic subgroup $P$ of $G$, the isomorphism $G/P\simeq (G\times G)/(P\times G)$
    provides an equivariant bi-Lipschitz bijection the limit set of
    $\Ga$ in $G/P$ and the limit set of $\Ga_\rho$ in $(G\times G)/(P\times G)$.
\end{Rmk}

We note that the global bi-H\"older condition in Proposition \ref{no} and Theorem
\ref{rigid2} can be relaxed to a local bi-H\"older condition by the following lemma.
\begin{lem} \label{global} Keep the notation as in Theorem \ref{rigid2} but assume $G_1$ and $G_2$ are semisimple, not just simple. Let $f:\La_{\theta_1}\to \La_{\theta_2}$ be a $\rho$-equivariant homeomorphism which is $\kappa$-bi-H\"older on some non-empty open subset  $U$ of $\La_{\theta_1}$ for some $\kappa>0$. Then
$f$ is $\kappa$-bi-H\"older globally.
\end{lem}
\begin{proof}  Let $\La_i=\La_{\theta_i}$ for $i=1,2$. Since $\La_1$ is $\Ga_1$-minimal, $\La_1=\Ga_1 U$ and hence, by compactness,
    we have $\La_1$ is a finite union of $\ga_k U$ for some $\ga_1, \cdots, \ga_n\in \Ga_1$. 
  If $f$ is not $\kappa$-H\"older globally, by the compactness of $\La_1$,
we have a sequence $\xi_i\to \xi$ and $\eta_i\to \eta $ such that
\be\label{inf} \frac{d_{\F_2}(f(\xi_i), f(\eta_i))}{d_{\F_1}(\xi_i, \eta_i)^\kappa}\to \infty.\ee 
Since $\F_2$ is compact, we have ${d_{\F_1}(\xi_i, \eta_i)}\to 0$.
Therefore, for some $1\le k\le n$, $\xi_i,\eta_i \in \ga_k U$ for all $i$.  Noting that the action of each element of
    $g_i \in G_i$ on $\F_i$ is a diffeomorphism for $i=1,2$,
     we can let $L$ be the maximum of the bi-Lipschitz constants of $\ga_k$ on $\F_1$ and of $\rho(\ga_k)$ on $\F_2$.
Now we have
$d_{\F_2}(f(\xi_i), f(\eta_i))\le L d_{\F_2}(f(\ga_k^{-1}\xi_i, \ga_k^{-1}\eta_i)) $
and $d_{\F_1}(\xi_i, \eta_i)\ge L^{-1} d_{\F_1}(\ga_k^{-1}\xi_i, \ga_k^{-1}\eta_i)) $. 
Since $f$ is $\kappa$-H\"older on $U$, it follows that  the ratio in \eqref{inf} is bounded, yielding a contradiction.
This shows that $f$ is $\kappa$-H\"older globally.
Similarly by considering $f^{-1}$, we can show that $f^{-1}$ is $\kappa^{-1}$-H\"older globally.
\end{proof}

Theorem \ref{m00} is now a special case
of the following corollary of Theorem \ref{rigid2} together with Lemma \ref{global}:
\begin{cor}
Let $\alpha_i$ be a simple root of $G_i$ for $i=1,2$.
 Suppose that there exists a $\rho$-equivariant bijection
 $f:\La_{\alpha_1}\to \La_{\alpha_2}$ which is
  $\kappa$-bi-H\"older on some non-empty open subset of $\La_{\alpha_1}$ for some $\kappa>0$.
Then $\kappa=1$ and $\rho$ extends to a Lie group isomorphism  $\bar\rho :G_1\to G_2$ which induces a diffeomorphism $\bar f:G_1/P_{\alpha_1}\to G_2/P_{\alpha_2}$ such that $\bar f|_{\La_1}=f$.
\end{cor}

Note that the conclusion $\kappa=1$ follows since $\bar f$ is diffeomorphism and hence bi-Lipschitz.

\begin{Rmk} \label{free} In general, we cannot replace $f$ bi-Lipschitz by Lipschitz in Theorem \ref{m00}. For example, let $\G$ be a Schottky subgroup of $\SL_2(\br)$ generated by two loxodromic elements $a, b$. Then for any $N\ge 2$,
the representation $\rho$ of $\Ga$ into $\SL_2(\br)$ given by $a\mapsto a^N$ and
$b\mapsto b^N$ induces an equivariant homeomorphism $\La\to \La$ which is
Lipschitz, but not bi-Lipschitz. Clearly, $\rho$ does not extend to $\SL_2(\br)$.
\end{Rmk} 

On the other hand, 
we have the following corollary of the proof of Theorem \ref{rigid2} where $f$ is required only to be Lipschitz under an extra hypothesis on the Hausdorff dimension of limit sets. In the statement below, a M\"obius transformation is the extension of
{\em any} isometry of $\bH^{n+1}$ to its boundary $\S^n=\partial\bH^{n+1}$.

\begin{cor} \label{lim}
     For $i=1,2$, let
      $\Ga_i$ be a convex cocompact Zariski dense subgroup of $G_i=\op{SO}^\circ(n_i+1, 1)$, $n_i\ge 1$. Let $\La_i\subset \S^{n_i} $ be the limit set of $\Ga_i$. Suppose that the Hausdorff dimension of $\La_1$
     is equal to the Hausdorff dimension of $\Lambda_2$.
     Let $f:\La_1\to \La_2$ be a $\rho$-equivariant homeomorphism which is Lipschitz on some non-empty open subset of $\La_1$.  Then $\rho$ extends to a Lie group isomorphism of $G_1\to G_2$ and $f$ extends to a M\"obius transformation of $\S^{n}$ for $n=n_1=n_2$.
\end{cor}

\begin{proof} By the proof of Lemma \ref{global}, $f$ is Lipschitz on all of $\La_1$.
Let $\Ga:=(\text{id}\times \rho)(\Ga_1)$ be
the self-joining subgroup of $G=G_1\times G_2$. 
For $i=1,2$, let $\alpha_i$ be the simple root of
$G=G_1\times G_2$ from the $i$-th factor. Then for any 
 loxodromic element 
$g=(\ga, \rho(\ga))\in G$, $\alpha_1(\lambda(g))$ and $\alpha_2(\lambda(g))$ 
are equal to $\lambda(\ga)$ and $\lambda(\rho(\ga))$ respectively.
Suppose that $\Ga$ is Zariski dense in $G$. 
The proof of Proposition \ref{no} for $\Ga$ shows that
if there exists a loxodromic element 
$g=(\ga, \rho(\ga)) \in \Ga$ such that $\alpha_1(\lambda(g))>\alpha_2(\lambda(g))$, then $f:\La_1\to \La_2$ cannot be Lipschitz. 
On the other hand, if $\La_1$ and $\La_2$
have the same Hausdorff dimension, the middle direction $(1,1)\in \fa\simeq\br^2$ is always contained in the interior of
the limit cone of $\Ga$ by \cite[Corollary 4.2]{KMO}. Note that when $\Ga_i$
are cocompact lattices and $n_1=n_2=2$, \cite[Corollary 4.2]{KMO} 
is due to Thurston \cite{Th}.
Therefore, the desired
element $g\in \Ga$ can always be found. This implies that $\Ga$ cannot be Zariski dense in $G$. As before, this implies the conclusion.
\end{proof}

\section{Slim limit sets of $G/P$ for $P$ non-maximal}
 Let $\Ga$ be a Zariski dense subgroup of a connected semisimple real algebraic group $G$. Fix a subset $\theta\subset \Pi$ with 
$\#\theta \ge 2$.
Recall from the introduction that  a subset $S\subset \F_\theta$
is called {\it slim}  if there exists a pair of distinct elements $\alpha_1$ and $\alpha_2$ of $\theta$ such that the limit set
$\La_\theta$ injects to $G/P_{\alpha_1}$
and $G/P_{\alpha_2}$ under the canonical projection map $\F_\theta \to G/P_{\alpha_i}$ for $i=1,2$.

 In this section we prove the following theorem.
 
 \begin{theorem}\label{sl}
    If $\# \theta \ge 2$ and $\La_\theta$ is a slim subset of $\F_\theta$,  then
 no non-empty open subset $U$ of $\La_\theta$ is contained in a proper $C^1$-submanifold of $\F_\theta$.  
\end{theorem} 

We first prove the following lemma which connects Theorem \ref{sl}
with Proposition \ref{no}.
\begin{lem}\label{dif} Let $\theta_0\subset \theta\subset \Pi$.
Suppose that $\La_\theta$ is a $C^1$-submanifold of $\F_\theta$ and that the canonical projection $\F_\theta\to \F_{\theta_0}$ 
is injective on $\Lambda_\theta$. Then 
  $\La_{{\theta_0}}$ is a $C^1$-submanifold
    of $\F_{{\theta_0}}$ and $f_{\theta_0}: \La_\theta\to \La_{\theta_0}$ is a $\Ga$-equivariant diffeomorphism.
\end{lem}

\begin{proof} For simplicity, we write $\La=\La_\theta$. We suppose that $\La$
is a $C^1$-submanifold of $\F_\theta$. 
Since the projection $\F_\theta\to \F_{\theta_0}$ given by $f(gP_\theta)=gP_{\theta_0}$ is a smooth map, 
its restriction  $f:\La\to \F_{\theta_0}$ is a $C^1$ map which is also injective by
hypothesis. 
We claim that there exists a point $x \in \Lambda$ where $df_x : T_x \Lambda \rightarrow T_{f(x)} \F_{{\theta_0}}$ is injective. Pick a point $x \in \Lambda$ which maximizes ${\rm rank} \, df_y$, $y\in \La$. Then there exists a neighborhood of $x$ in $\Lambda$ where $df$ has constant rank. Then if $r : = {\rm rank} \, df_x$, there exist local coordinates near $x$ where 
$$
f(x^1,\dots, x^m) = (x^1, \dots, x^r,0,\dots, 0).
$$
Since $f$ is injective, we must have $r=m$ and hence $df_x$ is injective. 

Now the set $\{ x \in \Lambda : df_x \text{ is injective}\}$ is open and $\Gamma$-invariant. Since $\Gamma$ acts minimally on $\Lambda$, this set must be all of $\Lambda$. Thus $f$ is an immersion. Since $f$ is an injective immersion and $\Lambda$ is compact, $f$ is a $C^1$-embedding. Hence $f$ is a diffeomorphism onto its image, which is $\La_{\theta_0}$.
In particular, $\La_{\theta_0}$ is a $C^1$-submanifold of $\F_{\theta_0}$.
\end{proof}

\subsection*{Proof of Theorem \ref{sl}}
By the hypothesis on the slimness of $\La_\theta$, there exists a pair of distinct elements $\alpha_1$ and $\alpha_2$ of $\theta$ such that
$\La_\theta$ injects to $G/P_{\alpha_1}$
and $G/P_{\alpha_2}$. 

Suppose on the contrary that some non-empty open subset $U$ of $\La_\theta$ is contained in some $C^1$-submanifold. Since $\La_\theta$ is $\Ga$-minimal, we have that
    for any $\xi\in \La_\theta$, $\Ga \xi$ is dense, so $\ga \xi\in  U$ for some $\ga\in \Ga$.  Since $\xi\in \ga^{-1}U$, it follows that $\La_\theta$ is a $C^1$-submanifold of $\F_\theta$. 
    By Lemma \ref{dif},
    we have $\G$-equivariant diffeomorphisms $f_{\alpha_i}:\La_\theta\to \La_{\alpha_i}$ for each $i=1,2$. Hence 
    $f_{\alpha_2}\circ f_{\alpha_1}^{-1}:\La_{\alpha_1}\to \La_{\alpha_2}$ is a $\Ga$-equivariant diffeomorphism, contradicting
    Proposition \ref{no}. This finishes the proof. 

\begin{Rmk} We remark that Proposition \ref{no} implies that 
if $\La$ is a slim subset of $G/P$, then there exists a maximal parabolic subgroup $Q$ containing $P$ such that
the projection $G/P\to G/Q$ restricted to $\La$  is not bi-Lipschitz.
\end{Rmk} 

  \subsection*{Antipodal groups} Theorem \ref{nm} applies to the class of $P$-antipodal discrete subgroups of $G$, which contains  any subgroup of a $P$-Anosov or a relatively $P$-Anosov subgroup. 
To define an antipodality, we recall that
 a parabolic subgroup $P$ is called reflexive if
its conjugacy class contains a parabolic subgroup $P'$ opposite to $P$, that is, $P\cap P'$ is a common Levi subgroup of both $P$ and $P'$.
For example, a minimal parabolic subgroup of $G$ is always reflexive.
For a parabolic subgroup $P$, let $P_{\op{reflexive}}$ be the largest reflexive parabolic subgroup contained in $P$. If $P=P_\theta$,
then $P_{\op{reflexive}}=P_{\theta\cup \i(\theta)}$.
\begin{Def}\label{anti}
    A discrete subgroup $\Ga$ is called \emph{$P$-antipodal} if its limit set in $G/P_{\op{reflexive}}$ is antipodal in the sense that any two distinct points are in general position.
    \end{Def}
    
    If a discrete subgroup $\Ga$ is $P$-antipodal, then its limit set on $G/P$ injects to $G/P'$ for any $P'$ containing $P$  \cite[Lemma 9.5]{KOW}. Hence if $\Ga$ is $P$-antipodal for a non-maximal parabolic subgroup $P$, then its limit set is a slim subset of $G/P$. Therefore the following corollary is a special case of Theorem \ref{nm}.

\begin{cor} \label{an} 
 Let $G$ be a connected semisimple real algebraic group of rank at least $2$ and $P$  a non-maximal parabolic subgroup of $G$.
     The limit set of a Zariski dense $P$-antipodal  subgroup of $G$ cannot be a $C^1$-submanifold of $G/P$.
\end{cor}
Note that there are many slim limit sets which are not antipodal (e.g., the limit set of a self-joining group defined in \eqref{self}).

\section{An example}\label{ex}
In this final section, we construct an example of a Zariski dense discrete subgroup of $\SL_8(\br)$ which explains the necessity of introducing $P_2'$ in the conclusion of Theorem \ref{m0} in the case when $P_2$ is not maximal.
The examples we construct are  Borel-Anosov and $(1,1,2)$-hyperconvex subgroups of $\SL_8(\br)$. 

We begin by setting up some notation. For any $d\ge 2$, let $A$ be the diagonal subgroup of
$\SL_d(\br)$ consisting of diagonal elements with positive entries so that $\fa$  and $\fa^+$ can respectively be identified with
$\fa=\{(u_1, \cdots, u_d): \sum_{k=1}^d u_k=0\}$ and
$\fa^+=\{(u_1, \cdots, u_d)\in \fa: u_1\ge \cdots \ge u_d\}$.
For $1\le k\le d-1$, let 
$$\alpha_k((u_1, \cdots, u_d))=u_k-u_{k+1};$$ then
$\Pi=\{\alpha_k:1\le k\le d-1\}$ is the set of all simple roots. For any $g\in \SL_d(\br)$, its Jordan projection $\lambda(g)\in \fa^+$ satisfies 
$$
\alpha_k(\lambda(g)) = \log \frac{\lambda_k(g)}{\lambda_{k+1}(g)}
$$
where $\lambda_1(g) \geq \cdots \geq \lambda_d(g)$ are the absolute values of the eigenvalues of $g$. Also, for $\theta \subset \Pi$, the boundary $\mathcal{F}_\theta = \SL_d(\br)/P_\theta$ coincides with the partial flag manifold consisting of flags with subspaces of dimensions $\{k : \alpha_k \in \theta\}$.  

Let $\Delta$ be a hyperbolic group and denote by $\partial \Delta$ its Gromov boundary. 
Recall from \cite{KP} that a representation $\rho:\Delta \to\SL_d(\mathbb R)$ is
$\{\alpha_k\}$-Anosov if there exist constants $c,C>0$ so that for all $\gamma\in \Delta$,
$$\alpha_k(\lambda(\rho(\gamma))\ge c|\gamma|-C$$
where $|\gamma|$ is
the minimal word length of $\gamma$ with respect to a fixed finite generating set of $\Delta$.
If $\rho$ is $\{\alpha_k\}$-Anosov, it admits a pair of unique  continuous equivariant embeddings $\xi_\rho^k:\partial \Delta \to \op{Gr}_k(\mathbb R^d)$
and $\xi_\rho^{d-k}:\partial \Delta\to \op{Gr}_{d-k}(\mathbb R^d)$. Furthermore, the image of $(\xi^k_\rho,\xi^{d-k}_\rho)$ coincides with the limit set of $\rho(\Delta)$ in $\mathcal{F}_{\{\alpha_k,\alpha_{d-k}\}}$. We say that $\rho$ is Borel-Anosov if it is $\{\alpha_k\}$-Anosov for all
$1\le k\le d-1$. The image of a Borel-Anosov representation is called a Borel Anosov subgroup.

A representation $\rho:\Delta \to \SL_d(\br)$ is 
$(1,1,2)$-hyperconvex  if it is $\{\alpha_1,\alpha_2\}$-Anosov and
for all distinct $x,y,z \in \partial \Delta$,
$$\xi_{\rho}^1(x) \oplus \xi_{\rho}^1(y) \oplus\xi_{\rho}^{d-2}(z)=\mathbb{R}^d.$$

Both being $\{\alpha_k\}$-Anosov and being $(1,1,2)$-hyperconvex are open conditions in the representation variety (see \cite[Proposition 6.2]{PSW}).

\medskip 

\begin{prop}\label{exam} There exists a Zariski dense discrete  subgroup $\Gamma<\SL_8(\br)$
which admits an equivariant Lipschitz bijection $\La_{\alpha_3}\to \La_{\alpha_1}$. Moreover, $\Gamma$ is Borel-Anosov, $(1,1,2)$-hyperconvex, and
the projection map $p:\La_{\{\alpha_1, \alpha_3\}}\to \La_{\alpha_3}$
is a bi-Lipschitz bijection.
\end{prop}
Theorem \ref{m0} in this case applies with $f=p^{-1}$, $P_1=P_{\alpha_3}$, $P_2=P_{\{\alpha_1, \alpha_3\}}$ and
$P_2'=P_{\alpha_3}$.

Let  $\Delta=\langle a_1,a_2\rangle$ be the free group with two generators $a_1, a_2$. 
Let $N\ge 2$.  
Let $\tau_1:\Delta\to \SL_2(\mathbb R)$ be a convex cocompact representation and  
$\tau_2:\Delta\to \SL_2(\mathbb R)$ 
be defined so that $\tau_2(a_i)=\tau_1(a_i)^N$ for $i=1,2$. We may choose $N$ large enough that 
$$   \frac{\alpha_1(\lambda(\tau_2(\ga)))}{\alpha_1(\lambda(\tau_1(\ga)))}\ge 4 \quad\text{for all non-trivial} \quad \gamma\in \Delta.$$

Let $\iota:\SL_2(\mathbb R)\to\SL_4(\mathbb R)$ be an irreducible 
representation, which is unique up to conjugations. Then each $\rho_i=\iota\circ\tau_i$ is a positive
representation and hence Borel Anosov and  $(1,1,2)$-hyperconvex 
\cite[Corollary 6.13]{PSW}. One easily checks that 
$\frac{\alpha_1(\lambda(\rho_2(\ga)))}{\alpha_1(\lambda(\rho_1(\ga)))}\ge 4 $ for all non-trivial $\gamma\in \Delta$. Then a theorem of Tsouvalas \cite[Theorem 1.9]{Tsouvalas} implies that
$f_{\rho_1, \rho_2}=\xi_{\rho_2}^1\circ ( \xi_{\rho_1}^{1})^{-1}$ is $4$-H\"older. Let $\Phi_0 : \Delta\rightarrow \SL_8(\mathbb{R})$ denote the representation given by the direct sum $\Phi_0 = \rho_1\oplus \rho_2$. 
One checks that 
\begin{align*}
 \lambda_1(\rho_2(\gamma)) > \lambda_2(\rho_2(\gamma))> \lambda_1(\rho_1(\gamma))>\cdots >\lambda_4(\rho_1(\gamma)) >  \lambda_3(\rho_2(\gamma))> \lambda_4(\rho_2(\gamma))
\end{align*}
for all non-trivial $\gamma\in \Delta$  and that $\Phi_0$ is 
Borel Anosov with limit maps given by
$$
\zeta_0^{k}(x) = 
\begin{cases}
\{0\} \oplus \xi_{\rho_2}^{k}(x) & \text{if } k=1,2 \\
\xi_{\rho_1}^{k-2}(x) \oplus \xi_{\rho_2}^{2}(x) & \text{if } k=3,4,5 \\
\mathbb{R}^4 \oplus \xi_{\rho_2}^{k-4}(x) & \text{if } k=6,7.
\end{cases}.
$$
Then, the fact that $f_{\rho_1,\rho_2}$ is $4$-H\"older implies that $\zeta_0^1\circ(\zeta_0^3)^{-1}$ is also $4$-H\"older. In particular, 
$\zeta_0^1\circ(\zeta_0^3)^{-1}:\Lambda_{\alpha_3}(\Phi_0(\Delta))\to\Lambda_{\alpha_1}(\Phi_0(\Delta))$ is Lipschitz.

However, $\Phi_0(\Delta)$ is not Zariski dense. 
Since $\Delta$ 
is the free group on two generators, 
there exists an arbitrary small deformation
$\Phi : \Delta \rightarrow \SL_8(\mathbb{R})$ of $\Phi_0$  which is Borel Anosov with Zariski dense image.
Arguing exactly as in~\cite[Section 9]{ZZ}, one can show that  
$\Phi_0$ and $\wedge^3 \Phi_0$ are both $(1,1,2)$-hyperconvex. 
Therefore, we may assume that $\Phi$ and $\wedge^3\Phi$ are both $(1,1,2)$-hyperconvex.  

One may then use standard techniques (cf. \cite{ZZ})
to show that if $\Phi$ is
sufficiently close to $\Phi_0$, then
$$\frac{2}{3}\le \frac{\alpha_1(\lambda(\Phi(\gamma)))}{\alpha_1(\lambda (\Phi_0(\gamma)))}\le \frac{3}{2}\quad \text{and}
\quad \frac{2}{3}\le \frac{\alpha_1(\lambda(\wedge^3\Phi(\gamma)))}{\alpha_1(\lambda(\wedge^3\Phi_0(\gamma)))}\le \frac{3}{2}$$
for all non-trivial $\gamma\in\Delta$.
Let $\zeta=(\zeta^k)$ be the limit map of $\Phi(\Delta)$ and 
$\hat\zeta_0^1:\partial\Delta\to\Lambda_{\alpha_1}(\wedge^3\Phi_0(\Delta))$
and
$\hat\zeta^1:\partial \Delta\to \Lambda_{\alpha_1}(\wedge^3\Phi(\Delta))$ be limit maps of $\wedge^3\Phi_0$ and $\wedge^3\Phi$.
One may again apply Tsouvalas's theorem  \cite[Theorem 1.9]{Tsouvalas} to conclude that
$\zeta^1\circ(\zeta_0^1)^{-1}$ and $\hat\zeta_0^1\circ(\hat\zeta^1)^{-1}$ are $\frac{2}{3}$-H\"older.
There is a $C^1$-equivariant identification of 
$\Lambda_{\alpha_1}(\wedge^3\Phi_0(\Delta))$ with 
$\Lambda_{\alpha_3}(\Phi_0(\Delta))$ 
and an analogous identification for $\Phi$, so we may conclude that 
$\zeta_0^3\circ (\zeta^3)^{-1}$ is $\frac{2}{3}$-H\"older. Now set $$\Gamma:=\Phi(\Delta)<\SL_8(\br).$$ Then the limit map
$$\zeta^1\circ(\zeta^3)^{-1}=\left(\zeta^1\circ(\zeta_0^1)^{-1}\right)\circ\left( \zeta_0^1\circ(\zeta_0^3)^{-1}\right)\circ \left( \zeta_0^3\circ (\zeta^3)^{-1}\right)$$
is a $\frac{16}{9}$-H\"older and hence yields a Lipschitz map from $\La_{\alpha_3}$ to $ \La_{\alpha_1}$. Since $\Gamma$ is Borel Anosov, the projection map
$\La_{\{\alpha_1, \alpha_3\}}\to \La_{\alpha_3}$ is now
a bi-Lipschitz homeomorphism.
This proves Proposition \ref{exam}.

\end{document}